\documentclass{article}
\usepackage{graphicx} 
\usepackage{mathrsfs}
\usepackage{amsmath}			
\usepackage{amsthm}			
\usepackage{amssymb}		
\usepackage[a4paper, total={6in, 9.5in}]{geometry}
\usepackage{csquotes}
\usepackage{hyperref}
\hypersetup{colorlinks=true, linkcolor=black, citecolor=black}
\usepackage{mathtools}
\usepackage{wrapfig}

\usepackage{algorithm}
\usepackage{algorithmic}

\usepackage{todonotes}

\newcommand{\trace}{\mathop\mathrm{tr}}

\makeatletter
\newcommand{\includesvggraphics}[2][\textwidth]{%
  \filename@parse{#2}%
  \includesvg[inkscapepath=svg-inkscape/\filename@area,#1]{#2}%
}
\makeatother

\theoremstyle{plain}

\newtheorem{theorem}{Theorem}
\newtheorem{lemma}{Lemma}
\newtheorem{corollary}{Corollary}
\newtheorem{remark}{Remark}

\newcommand\norm[1]{\lVert#1\rVert}
\newcommand{\R}{\mathbb{R}}
\newcommand{\mc}[1]{\mathcal#1}
\newcommand{\dkl}[2]{D_{\rm KL}\left(#1 \,\|\, #2\right)}

\title{Posterior sampling with Adaptive Gaussian Processes in Bayesian parameter identification}
\author{Paolo Villani, Daniel Andrés-Arcones, Jörg Unger, and Martin Weiser}

\begin{document}

\maketitle

\begin{abstract}
Posterior sampling by Monte Carlo methods provides a more comprehensive solution approach to inverse problems than computing point estimates such as the maximum posterior using optimization methods, at the expense of usually requiring many more evaluations of the forward model. Replacing computationally expensive forward models by fast surrogate models is an attractive option. However, computing the simulated training data for building a sufficiently accurate surrogate model can be computationally expensive in itself, leading to the design of computer experiments problem of finding evaluation points and accuracies such that the highest  accuracy is obtained given a fixed computational budget. Here, we consider a fully adaptive greedy approach to this problem. Using Gaussian process regression as surrogate, samples are drawn from the available posterior approximation while designs are incrementally defined by solving a sequence of optimization problems for evaluation accuracy and positions. The selection of training designs is tailored towards representing the posterior to be sampled as good as possible, while the interleaved sampling steps discard old inaccurate samples in favor of new, more accurate ones.
Numerical results show a significant reduction of the computational effort compared to just position-adaptive and static designs.
\end{abstract}

\section{Introduction}

The response of many parameter-dependent models is described by a function $y$ mapping parameters $p$ to observable quantities $y(p)$. The inverse problem of inferring the posterior probability distribution $\pi_{\rm post}$ from measurements $y^m$ is frequently encountered in various fields including physics, engineering, finance, and biology~\cite{KaipioSomersalo2005}.  In contrast to point estimates $p_*$, e.g., the maximum likelihood estimate $\mathrm{arg\,min}_p \|y(p)-y^m\|$, which are often computed by optimization methods~\cite{EngelHankeNeubauer1996}, a comprehensive view on the likely parameters and their uncertainty is provided by sampling the posterior probability distribution by Markov Chain Monte Carlo (MCMC) methods such as Metropolis-Hastings~\cite{Hastings1970} and Ensemble Sampling~\cite{GoodmanWeare}. The drawback is the large number of forward evaluations required for a faithful representation of the posterior density or an accurate approximation of statistical moments such as mean or variance, usually exceeding several thousands of evaluations.

Often, evaluating $y$ involves a complex numerical procedure, such as the numerical solution of a partial differential equation. Posterior sampling for parameter estimation and uncertainty quantification is then computationally expensive and often inhibitive. To reduce the number of model evaluations, different approaches have been developed. Delayed Acceptance MCMC~\cite{ChristenFox2005} evaluates the forward model only on points which are likely to be accepted. An alternative is provided by fast surrogate models approximating $y$ and utilized as a replacement for the forward model when sampling the posterior, in particular if both parameters $p$ and measurements $y^m$ are low-dimensional. Various types of surrogates are employed, including polynomials, sparse grids, tensor trains, artificial neural networks, and Gaussian process regression (GPR)~\cite{Schneider,Zaytsev}, on which we focus here. 

The construction of a sufficiently accurate surrogate model can require itself a large number of model evaluations $y(p_i)$. To address this issue, the choice of both the number and the position of evaluation point must be carefully considered. Various strategies to select training points have been proposed, in particular for analytically well-understood GPR~\cite{RasmussenWilliams2006}. 

When measurements $y^m$ are not yet available, the training of the surrogate must rely on information on the parameter space and the model only. Space-filling designs seek the training points to be occupying most of the parameter space and can be selected a priori~\cite{Giunta,Queipo} or adaptively~\cite{Crombecq,Joseph,Lehmensiek,Sugiyama}. Active learning relies on pointwise estimates of the surrogate approximation error for guiding the selection process, usually including the parameter point into the training set, that maximizes some acquisition or utility function~\cite{Hennig,Mockus,Wu}. 

If instead measurement data $y^m$ is available at training time, a goal-oriented strategy can additionally improve the efficiency of the training design construction~\cite{Dinkel2024,SinsbeckNowak2017,WangBroccardo2020}, by aiming not at global accuracy on the whole parameter space, but at a faithful representation of the posterior distribution. Such an approach can reduce the computational effort for building the surrogate model significantly, and train the surrogate while samples from the posterior distribution are drawn. These posterior-based surrogating strategies can guarantee asymptotic convergence to the posterior in the Hellinger distance~\cite{HelinStuartTeckentrupZygalakis2023}.
In addition, a reduced complexity of the surrogate in terms of the number of training points leads to a faster surrogate evaluation. 

Whenever numerical procedures like finite element (FE) solvers are employed to evaluate the forward model, the evaluations $y(p_i)$ will not be exact due to discretization and truncation errors. With adaptive mesh refinement techniques, the discretization error, and, in turn, the computational effort, can be chosen arbitrarily in a wide range by specifying an evaluation tolerance. While uniformly high accuracy can be used, this incurs a high computational effort. The trade-off between evaluation accuracy and computational cost is often not considered, assuming an accurate-and-costly numerical discretization of the model as ground truth. Significant efficiency gains by joint optimization of evaluation position and tolerance given a limited computational budget have been achieved for pure GPR~\cite{SemlerWeiser2023} as well as for gradient enhanced GPR~\cite{SemlerWeiser2024} in the context of maximum posterior estimates and a completely offline surrogate construction.

In this paper, we provide a detailed investigation of a recently proposed surrogate-based sampling strategy for Bayesian inverse problems with interleaved goal-oriented surrogate model training~\cite{VillaniUngerWeiser2024}. While samples are accumulated, the surrogate model is improved by including new training points and decreasing the tolerance of existing ones. These choices are quantified by minimization of surrogate error weighted by the current posterior representation. We extend on~\cite{VillaniUngerWeiser2024} by giving a comprehensive presentation of the complete adaptive strategy, considering different error metrics for measuring the quality of the posterior approximation, and evaluating the strategy in more detail and on more test problems, 

The remainder of the paper is structured as follows: In Sec.~\ref{sec:IP} we briefly recall Bayesian inversion introducing notation. GPR surrogate models and surrogate-based inverse problems are discussed in Sec.~\ref{sec:GPinversion}. The different error metrics and the sampling interleaved with adaptive surrogate construction are worked out in Sec.~\ref{sec:strategy}. Numerical experiments are presented in Sec.~\ref{sec:experiments}.


\section{Bayesian inversion via Posterior Sampling}
\label{sec:IP}

In this section, we will define the setting of our problem and present some key techniques.

\subsection{Forward model} \label{sec:fedef}

We start by considering a forward model
\[ y : \Omega \subset \R^d  \to \R^m, \quad p \mapsto y(p), \]
which to any given parameter value $p$ in a domain $\Omega$ assigns a value in the measurement space $\R^m$. We assume the forward model to be approximated by a computational model $y_\tau:\Omega\to \R^m$ parameterized by a tolerance $\tau > 0$, such that $\|y(p)-y_\tau(p)\|\le \tau$ is satisfied, with $\|\cdot\|$ denoting an arbitrary, possibly problem-dependent, norm on $\R^m$.

\subsection{Bayesian inverse problem} \label{sec:invprobdef}

We suppose now to have some measurements $y^m$ taken from a real world phenomenon that we want to represent by the forward model $y(p_*)$ for the true but unknown parameter $p_*$. We assume additive normally distributed measurement errors $\eta$, such that
\[
y^m = y(p_*) + \eta, \quad\text{with } \eta \sim \mathcal{N}( {0}, \Sigma^2)
\]
holds for zero error mean and covariance $\Sigma$.  
Now, due to the preceding assumptions we obtain the likelihood distribution
\[
\pi (y^m \mid p) \sim \mathcal N (y(p), \Sigma^2 ),
\] 
i.e.\ the measurements given the parameter value are normally distributed. Neglecting the normalization constant not needed for MCMC sampling, we obtain the computable likelihood
\begin{equation}\label{eq:likelihood}
L(p) = \exp \Big (-\frac{1}{2}\norm{y^m - y(p)}_{\Sigma^{-2}}^2 \Big ) \propto \pi (y^m \mid p),
\end{equation} 
where, for any SPD matrix $M$, $\norm{x}_M = x^T M x $ is the norm induced by $M$.
Note that evaluating the likelihood requires evaluating the forward model $y$

Adopting a Bayesian point of view, we treat unknown quantities as random variables: in this perspective, the parameter $p$ is a random vector with a prior distribution $\pi(p)$ representing the information we have before measurements are taken.

By the theorem of Bayes, we can derive the posterior distribution $\pi$, which can be interpreted as the solution of the inverse problem. Its density is given by
\begin{equation} \label{eq:bip}
    \pi(p \mid y^m ) =
    \frac{\pi(p) \ \pi(y^m \mid p )}{\pi(y^m)} \propto \pi(p) \ L(p).
\end{equation}
 The normalization constant or evidence $\pi(y^m)^{-1}$ is in general computationally unavailable, as it requires an integration over the often high-dimensional parameter space $\Omega$, but is fortunately not always needed.

\subsection{Posterior sampling}

We aim at solving the inverse problem
by sampling the posterior distribution. Representing $\pi(p\mid y^m)$ by a sufficiently large sampling allows estimating statistical moments by, e.g., the sample mean and sample covariance, and thus provides a direct uncertainty quantification.

To sample the posterior, we utilize Markov Chain Monte Carlo (MCMC) methods, a class of sampling techniques commonly employed to sample from distributions. The resulting samples are not independent one from the other, but in sufficient number provide a reliable approximation of the posterior distribution.

MCMC methods have the advantage of not needing normalizing constants for the target distribution, i.e.\ the distribution needs to be known only up to a constant such as in~\eqref{eq:bip}. 
However, at each step they need to evaluate in a new parameter point the product between prior and likelihood, i.e. the unnormalized posterior, in order to compute a rejection probability. By this, sampling methods need a considerable number of forward model evaluations, usually far greater than the one needed for solving the optimization problems of maximum posterior or maximum likelihood point estimates.

Given samples $\{ p_1, \dots, p_N\}$ from the posterior, we can approximate integrals weighted by the posterior by summing contributions from the samples:
\begin{equation} \label{eq:MC-integration}
\int f(p) \ \pi(p \mid y^m) \ dp \approx \frac{1}{N} \sum_{i=1}^N f(p_i).
\end{equation}
This is a case of Monte Carlo integration, and provides a viable approach for the numerical estimation of integrals which involve a distribution only known up to a constant. It is particularly efficient in the case where multiple integrals which involve the same distribution are to be computed and we will use it in such a case.

\section{Inversion with Gaussian Process Surrogates}
\label{sec:GPinversion}

The inverse problem formulated in Section \ref{sec:invprobdef} could in theory be solved using the mathematical model $y$ to evaluate the likelihood. In practice, one usually can only evaluate an approximate model $y_\tau$ for some tolerance $\tau>0$, at some computational cost $W(\tau)$. Such costs often make solving the problem accurately undesirable or even infeasible, especially with costly techniques such as MCMC. A common solution to this problem is introducing a cheap surrogate model. A surrogate model needs to be both sufficiently accurate and relatively inexpensive to evaluate when compared with the full model. In this paper, we consider Gaussian Process Regression as surrogate models.

\subsection{Gaussian Process Regression}\label{sec:GPR}

Gaussian process Regression (GPR) is an approximation technique that can efficiently learn large classes of functions~\cite{RasmussenWilliams2006}, providing a . 

We consider an experimental design $\mc D = \left( (p_i,\tau_i)\right)_{i=1,\dots,s}$, with evaluation parameter positions $p_i$  and evaluation tolerances $\tau_i$. We then consider inexact evaluation data corresponding to $\mc D$, $Y_T =(y_i)_{i=1,\dots,s}$ with $y_i \approx y(p_i)$. We will assume the actual evaluation errors $e_i = y_i-y(p_i)$ to be independently and normally distributed with mean zero and standard deviation $\tau_i$. In case of multi-dimensional model output for $m>1$, we also assume the error components to be independently and identically distributed.

From a Bayesian perspective, we look for a surrogate model $y_{\mathcal{D}} \approx y$ that is constructed or identified from the design $\mc D$. Under the assumptions regarding the training data, the regression likelihood of evaluating $Y_T$ if the true response is $Y=(y(p_i))_{i=1,\dots,s}$ is 
\[
\pi_{\rm like}(Y_T \mid Y) \propto \exp\left(-\frac{1}{2}\|Y_T - Y\|^2_{\mathfrak{T}^{-2}}\right)
\]
with the diagonal covariance $\mathfrak{T} ^2 = I_m \otimes \mathrm{diag}(\tau_1^2,\dots,\tau_s^2)  $.

\begin{remark}
    The assumption of independently and normally distributed evaluation errors is for sure not satisfied in the case of usual discretization errors from, e.g., finite element approximations, but allows a straightforward analytical treatment and derivation of adaptive algorithms. 
    
    The assumption of the error components being identically and independently distributed is in general also wrong. Instead of the $m$-dimensional unit matrix $I_m$, an arbitrary correlation matrix could have been used. 
\end{remark}

The regression prior regularity assumption on the forward model $y$ is, that it is a realization of a Gaussian process. A Gaussian process  $\mc G$ on $\Omega$ is a stochastic process $\{\mc{G}(p)\}_{p \in \Omega}$ such that, for any  $p_1, \dots, p_k \in \Omega$, $( \mc G (p_1), \dots, \mc G (p_k) ) $ has a multivariate normal distribution. A Gaussian process is completely determined by its mean function $\mu: \Omega \rightarrow \R^m$ and its covariance function or kernel $k: \Omega \times \Omega \rightarrow \R^{m\times m}$, and thus denoted as $GP (\mu, k )$. For a family of values $Y = (y(p_1),\dots,y(p_s),y(p'))\in\Omega^{s+1}$, the regression prior probability is therefore 
\[
    \pi_{\rm prior}(Y) \propto \exp\left(-\frac{1}{2}\|Y-M\|_{K^{-1}}^2\right)
\]
with mean $M=(\mu(p_1),\dots,\mu(p_s),\mu(p'))$ and the covariance 
\[
    K = \begin{bmatrix} k(p_1,p_1) & \dots & k(p_1,p') \\
                        \vdots               & \ddots & \vdots \\
                        k(p',p_1) & \dots & k(p',p')
    \end{bmatrix}.
\]
By Bayes' theorem, the posterior probability of values $Y$ given that inexact evaluations $Y_T$ are available, is then
\[
    \pi_{\rm post}(Y | Y_T) = \frac{\pi_{\rm like}(Y_T\mid Y) \pi_{\rm prior}(Y)}{\pi(Y_i)} \propto \exp\left( -\frac{1}{2}\|Y-\bar Y\|_{\Gamma^{-1}}^2\right)
\]
with covariance $\Gamma = \left(K^{-1}+\mathfrak{T}^{-2}\right)^{-1}$ and mean $\bar Y = \Gamma(K^{-1}M +\mathfrak{T}^{-2} Y_\tau)$. 

Regression is performed by including a parameter $p_{n+1}$ into the set for which no evaluation is known, i.e. formally $\tau_{s+1} = \infty$. Then, its marginal posterior distribution is 
\begin{equation}
    y_{\mc D}(p_{s+1}) \sim \mc N(\bar y(p_{s+1}),\Gamma(p_{n+1})),
\end{equation} 
with $\bar y(p_{s+1}) = \bar Y_{s+1}$, $\Gamma(p_{s+1}) = \Gamma_{s+1,s+1}$, and indices interpreted as block ranges in case $m>1$. Note that $\bar y (p)$ is linear in $Y$ and one can easily compute its gradient by computing $\nabla_\theta k$. For a detailed exposition we refer to~\cite[chap. 2.2]{RasmussenWilliams2006}.

The choice of an appropriate kernel is crucial for the quality of a GP surrogate, since it determines key features of the GP, such as the smoothness of its realizations. In fact, the choice of the covariance can be seen as the assumption that the target function lives in a certain Hilbert space, namely the Reproducing Kernel Hilbert Space corresponding to the kernel $k$~\cite[chap. 4]{ChristmannSteinwart2008}. There is a vast choice of kernels: an overview for the single-output case, i.e. $m=1$, can be found in~\cite[chap. 2]{Duvenaud} and~\cite[chap. 4]{RasmussenWilliams2006}, while for the multiple-output case a review is offered by~\cite{AlvarezRosascoLawrence2012}; in the context of surrogates for linear PDE models,~\cite{BaiTeckentrupZygalakis2024} proposes a kernel construction procedure for PDE-informed priors.  

Kernels are usually assumed to be stationary, i.e. $k(p, p') = g(\norm{p-p'}_k)$  for some norm $\norm{\cdot}_k$ and function  $g:\R \to \R^{m\times m}$. In the multiple output case it is common to consider separable kernels $k(p, p') = \kappa(p, p') K_c$ for a scalar kernel $\kappa: \Omega \times \Omega \rightarrow \R$  and symmetric positive definite $K_c \in R^{m\times m}$, where $\kappa$ models the correlation between different parameters and $K_c$ models the correlation between different components.

An elementary and frequent choice is the separable stationary Gauss kernel
\begin{equation}\label{eq:gauss-kernel}
     k(p, p') =  \exp ( - \norm{p - p'}^2_{L} ) \ K_c,
\end{equation}
which yields $C^\infty$ realizations. Note that the kernel depends on the component covariance matrix $K_c$ and the length scale matrix $ L \in \R^{d \times d} $. These are known as the GP's hyperparameters and their adaptation to the data to be approximated is a crucial step in GP regression~\cite[chap. 2.3, 5]{RasmussenWilliams2006}.

\subsection{Likelihood representation} \label{sec:likelihoods}

We will postpone the question of how to select the design $\mc D$ of training points and their tolerances to Sec.~\ref{sec:doe} and for now assume it to be given, defining a normal predictive distribution $y_{\mc D}$ of mean $\bar y $ and covariance $\Gamma$. 

that is chosen as surrogate for the forward model $y$. Turning towards the original inverse problem formulated in terms of the posterior distribution~\eqref{eq:bip}, the exact model $y$ needs to be replaced by the surrogate $y_{\mc D}$ in the likelihood $L$ from~\eqref{eq:likelihood}.

The straightforward choice of interpreting the conditional mean $\bar y$ as a deterministic response surface to replace the forward model $y$ in the likelihood~\eqref{eq:likelihood} yields 
\[
L_{\text{plug-in}}(y^m \mid p, \bar y) = \exp \Big (-\frac{1}{2}\norm{y^m -\bar y(p)}_{\Sigma^{-2}}^2 \Big ).
\]

However, a Gaussian process $y_{\mc D}$ as a stochastic surrogate model does not only provide its mean, but also the whole distribution. The likelihood of observing $y^m$ thus depends not only on $\bar y$, but also on the variance. The full likelihood is the marginal distribution of the joint distribution of $y^m$ and $y_{\mc D}(p)$, i.e. the sum of measurement error and surrogate predictive distributions:
\begin{align}
    L_{\mc D}(p) := \pi_\mc D (y^m \mid p, y_{\mc D})
    &=  \mathbb E_{y_{\mc D}(p)} \Big [ \exp \Big (-\frac{1}{2}\norm{y^m -y_{\mc D}(p)}_{\Sigma^{-2}}^2 \Big ) \Big ] \notag \\
    &= (2\pi)^{-\frac{m}{2}} \text{det}\big ( \Sigma^{2} +\Gamma(p) \big )^{-\frac{1}{2}} \exp \Big( -\frac{1}{2}\norm{y^m - \bar y (p)}_{(\Sigma^{2} + \Gamma(p))^{-1}}^2 \Big ) . \label{eq:full-likelihood}
\end{align}
Both choices can be justified, and interpreted as $L^1$ loss and $L^2$ loss, respectively~\cite{ JarvenpaaGutmannVehtariMarttine2021,SinsbeckNowak2017}. Here, we use the fully stochastic likelihood~\eqref{eq:full-likelihood}, since including the surrogate variance increases the likelihood variance and avoids an overconfident posterior.
The choice of such a likelihood comes at the price of having a more complicated likelihood function, since now the covariance depends on $p$. For $\Gamma(p) > \Sigma^2$, this can introduce more local minima and non-convex behavior, which can, however, be expected to be less of a problem for sampling based approaches than for optimization methods.

Finally, the fully Bayesian posterior density, informed about the surrogate predictive variance, is
\begin{equation} \label{eq:posterior}
\pi(p \mid  y^m,\mc D) = \frac{ \pi(p) \ L_\mc D(p)}{\pi_{\mc D}(y^m)}.
\end{equation}

\section{Surrogate-based solution of the Inverse Problem.} \label{sec:strategy}

In practice, sampling the true posterior $\pi(p\mid y^m)$ as given in~\eqref{eq:bip} is not possible and the surrogate-based posterior $\pi(p \mid  y^m,\mc D)$ defined in~\eqref{eq:posterior} provides the currently available approximation. The selection of a good training design is therefore crucial for an adequate surrogate accuracy and a faithful posterior representation.

The local surrogate accuracy can be controlled by the evaluation design $\mc D$, including evaluation point positioning $p_i$ and evaluation tolerances $\tau_i$, and affects the overall accuracy $E(\mc D)$ as well as the computational effort $W(\mc D)$. Following~\cite{SemlerWeiser2023,VillaniUngerWeiser2024}, we aim at selecting the design by minimizing the posterior approximation error given a certain computational budget:
\begin{equation} \label{eq:doe}
    \min_{\mc D} E(\mc D) \quad \text{subject to} \quad W(\mc D) \le W.
\end{equation}
For that, we define subsequently both an error model $E(\mc D)$ and a work model $W(\mc D)$.

\subsection{Error models} \label{sec:error-model}

Here and in the following, we will distinguish between the error metric $\mc E (\mc D)$, which is a target quantity that depends on the forward model $y$ and/or the true posterior $\pi(p \mid y^m)$, and the error model $E (\mc D)$, which is an approximation of $\mc E(\mc D)$ that does not depend neither on the forward model nor on the true posterior, being thus computationally available. The error models we will consider are of the form 
\[
E(\mc D) = \int_\Omega \pi(p \mid y^m, \mc D) e_\mc D (p) \, dp,
\]
with the local error indicator $e_\mc D$ being some function of $p$ depending on the Gaussian process $y_\mc D$ and thus on the design $\mc D$. In the rest of this work, we will add superscripts when discussing specific error models, otherwise, $E$ and $e$ are assumed to be any error model and indicator of this kind.

Different types of error metrics $\mc E$ are conceivable, resulting in different interpretations of the training problem~\eqref{eq:doe}:
\begin{itemize}
    \item Error metrics which quantify the distance between the true posterior $\pi(p \mid y^m )$,~\eqref{eq:bip}, and the surrogate-based posterior $\pi_{\mc D}(p \mid  y^m) $,~\eqref{eq:posterior}. As a metric of this kind, we will consider the Kullback-Leibler divergence between the two posteriors: 
    \[ 
    \dkl{\pi(\cdot \mid y^m)}{\pi( \cdot \mid  y^m,\mc D)} = \int_{\Omega} \pi(p \mid y^m) \log \frac{\pi(p \mid y^m)}{\pi(p \mid  y^m,\mc D)} \, dp.
    \]

    \item Error metrics which quantify the surrogate error in the posterior region. As a metric of this kind, we will consider the posterior-weighted $L^2$ error between the forward model $y$ and the surrogate mean $\bar y$:
    \[  
    \norm{y-\bar y}^2_{L^2 \left (\Omega; \pi(\cdot \mid y^m) \right )}= \int_{\Omega} \pi(p \mid y^m)  \norm{y(p) - \bar y (p) }_2^2 \, dp .
    \]
    
\end{itemize}
These classes of metrics are correlated, but not identical: the first is more general, as it assumes the full posterior distribution as a target, while the second targets directly the GP mean's error and only as a consequence of this the posterior.

\subsubsection{Kullback-Leibler divergence} \label{sec:dKL}
In this section, we consider the Kullback-Leibler divergence as posterior approximation error metric
\[
    \mc E^{\rm KL}(\mc D) = \dkl{\pi(\cdot \mid y^m)}{\pi( \cdot \mid  y^m,\mc D)},
\]
describing the information gain when using the true posterior~\eqref{eq:bip} for inference instead of the approximate posterior~\eqref{eq:posterior}. As the true posterior is unknown, we will replace it by an estimate based on the local surrogate mean $\bar y$ and its variance $\Gamma$ to obtain the corresponding error model.

\begin{theorem} \label{thm:dKL-bound}
    Assume that $\pi_{\mc D}(y^m) \le \alpha \pi(y^m)$ holds for some $\alpha<\infty$. Let  $\phi(a,b) = \frac{1}{2}a + b\sqrt{a}$, and
    \begin{equation}\label{eq:psi}
    \psi(y^m,p) = \frac{1}{2} \trace\left(\Sigma^{-2}\Gamma(p) \right) + \phi\left(\|\bar y(p)-y(p)\|_{\Sigma^{-2}}^2,\|y^m - y(p)\|_{\Sigma^{-2}}\right)
    + \log\alpha.
    \end{equation}
    Then,
    \begin{align}
    \mc E^{\rm KL}(\mc D) 
    &\le \int_\Omega \pi(p\mid y^m) \psi(y^m,p)  \, dp \label{eq:target-a}\\
    &\le \int_\Omega \pi (p\mid y^m, \mc D) \psi(y^m,p) \exp(\psi(y^m,p)) \, dp \label{eq:target-b} 
    \end{align}
    holds.
\end{theorem}

Before proving the result, let us briefly discuss the individual terms in the relevant quantity $\psi$ for the surrogate converging to the true forward model, i.e. $\|\Gamma(p)\|\to 0$. As $\Sigma$ is fixed, the first term converges linearly to zero. The second term diminishes due to $\bar y\to y$ in a sense given in Lemma~\ref{lem:surrogate-error-mean} below. Finally, for $\Gamma\to 0$ and $\bar y\to y$, we also expect $\alpha \to 1$.  Thus, the upper bound~\eqref{eq:target-b}, while not being sharp, captures the convergence behavior of the surrogate and is therefore a reasonable error quantity to consider.

\begin{proof}
We start with bounding
\begin{align*}
    2\log \frac{\pi(p \mid y^m)}{\pi( p \mid y^m, \mc D)}
    &= 2\log \left( \frac{\pi(p)\pi(y^m\mid p)}{\pi(y^m)} \frac{\pi_{\mc D}(y^m)}{\pi(p)L_{\mc D}(p)} \right) \\
    &= 2\log\left( \frac{\pi(y^m\mid p)}{\pi(y^m)} \frac{\pi_{\mc D}(y^m)}{\pi_{\mc D}(y^m \mid p, \mc D)} \right)  \\
    &= \log\frac{\det(\Sigma^2+\Gamma(p))}{\det(\Sigma^2)} - \|y^m - y(p)\|_{\Sigma^{-2}}^2 + \|y^m-\bar y(p)\|_{(\Sigma^2+\Gamma(p))^{-1}}^2 + 2 \log \frac{\pi_{\mc D}(y^m)}{\pi(y^m)}.
\end{align*}
Let us treat the terms individually. First, simple linear algebra yields
\[
    \log\frac{\det(\Sigma^2+\Gamma(p))}{\det(\Sigma^2)} = \log \det\left(I_m + \Sigma^{-2}\Gamma(p) \right)\le m \log \left( \frac{1}{m}\trace\left( I_m + \Sigma^{-2}\Gamma(p)\right)\right)
    \le \trace\left(\Sigma^{-2}\Gamma(p) \right).
\]
By assumption, we then have 
\[
2 \log\frac{\pi_{\mc D}(y^m)}{\pi(y^m)} \le 2\log\alpha
\]
and finally
\begin{align*}
    - \|y^m - y(p)\|_{\Sigma^{-2}}^2 &+ \|y^m-\bar y(p)\|_{(\Sigma^2+\Gamma(p))^{-1}}^2 \\
    &\le - \|y^m - y(p)\|_{\Sigma^{-2}}^2 + \|y^m-\bar y(p)\|_{\Sigma^{-2}}^2 \\
    &= \left(y^m-y(p)+y^m-\bar y(p)\right)^T \Sigma^{-2} \left(y^m-y(p)-y^m+\bar y(p)\right) \\
    &\le \|2y^m - (y(p)+\bar y(p))\|_{\Sigma^{-2}} \|\bar y(p)-y(p)\|_{\Sigma^{-2}} \\
    &\le \left(2\|y^m-\bar y(p)\|_{\Sigma^{-2}} + \|\bar y(p)-y(p)\|_{\Sigma^{-2}}\right) 
    \|\bar y(p)-y(p)\|_{\Sigma^{-2}} \\
    &= 2\phi(\|\bar y(p)-y(p)\|_{\Sigma^{-2}}^2,\|y^m-\bar y(p)\|_{\Sigma^{-2}}).
\end{align*}
In sum, we obtain 
\begin{align}
    \log \frac{\pi(p \mid y^m)}{\pi(p\mid y^m, \mc D)}
    &\le  \frac{1}{2} \trace\left(\Sigma^{-2}\Gamma(p) \right) + \phi\left(\|\bar y(p)-y(p)\|_{\Sigma^{-2}}^2,\|y^m- \bar y(p)\|_{\Sigma^{-2}}\right)
    + \log\alpha \notag \\
    &\le \psi(y^m,p) \label{eq:log-ratio}
\end{align}
and thus the claim~\eqref{eq:target-a}. Observing $\pi(p\mid y^m) 
\le \pi(p\mid y^m, \mc D)\exp(\psi(y^m,p))$ by~\eqref{eq:log-ratio} yields the claim~\eqref{eq:target-b}.
\end{proof}

The upper bound~\eqref{eq:target-b} is still not computable efficiently, as the exact forward model $y$ enters $\psi$ in~\eqref{eq:psi}. Strictly speaking, as the forward model $y$ is assumed to be a (realization of  a) Gaussian process, the posterior $\pi(p\mid y^m)$, the Kullback-Leibler divergence $\mc E (\mc D)$, and the upper bound~\eqref{eq:target-b} are themselves random variables. For a computationally available estimate, we derive an upper bound on the mean of~\eqref{eq:target-b}.

\begin{lemma} \label{lem:surrogate-error-mean}
    Let $y(p) \sim \mc N( \bar y(p),\Gamma(p))$.
    Then, the mean deviation is 
    \[
        \mathbb{E}\left[ \|\bar y(p) - y(p)\|_{\Sigma^{-2}}^2\right] = \trace\left(\Sigma^{-2}\Gamma(p) \right),
    \]
    and 
    \[
        \mathbb{E}\left[\phi(\|\bar y(p) - y(p)\|_{\Sigma^{-2}}^2,\|y^m-\bar y(p)\|_{\Sigma^{-2}})\right] \le \phi(\trace\left(\Sigma^{-2}\Gamma(p) \right),\|y^m-\bar y(p)\|_{\Sigma^{-2}})
    \]
    holds.
\end{lemma}
\begin{proof}
    By assumption $\bar y(p) - y(p) = \Gamma(p)^{1/2}z$ with $z\sim \mc N(0,I)$. Consequently, with $S\in\R^{m\times m}$ being the diagonal matrix of eigenvalues of $\Sigma^{-1}\Gamma(p)^{1/2}$, we obtain 
    \begin{align*}
        \|\bar y(p) - y(p)\|_{\Sigma^{-2}}^2
        = \|\Sigma^{-1}\Gamma(p)^{1/2} z \|^2 
        = \| S z \|^2  
        = \sum_{i=1}^m s_i^2 z_i^2.
    \end{align*}
    As $z_i^2$ is $\chi^2$-distributed with mean 1, the first claim follows from the trace being invariant with respect to similarity transforms, i.e.\ $\trace(S) =\trace\left(\Sigma^{-2}\Gamma(p) \right)$. The second claim is a direct consequence of $\phi$ being concave.
\end{proof}

\begin{remark}
    When performing GPR, we assumed a priori that $y$ is a realization of a Gaussian Process $GP(\mu, k)$ and then inferred a posterior predictive distribution $y_\mc D(p)$ for $y(p)$ given the training data, which satisfies the condition in Lemma~\ref{lem:surrogate-error-mean}. Thus, the assumptions required by GPR imply that Lemma~\ref{lem:surrogate-error-mean} can be applied.
\end{remark}
By the above results, as an error model for the Kullback-Leibler divergence error metric we adopt the GP-averaged upper bound
\begin{equation} \label{eq:dKL-model}
\begin{gathered}
    E^{\rm KL}(\mc D) = \int_\Omega \pi (p\mid y^m,\mc D) e^{\rm KL}_\mc D(p) \, dp, 
\quad \text{with} \quad
    e^{\rm KL}_\mc D(p) =  \bar\psi(p) \exp(\bar\psi(p)) \quad\text{and} \\
    \bar\psi(p) =\trace\left(\Sigma^{-2}\Gamma(p) \right) + \|y^m-\bar y(p)\|_{\Sigma^{-2}}\sqrt{\trace\left(\Sigma^{-2}\Gamma(p) \right)}
    + \log\alpha,
\end{gathered}
\end{equation}
which no longer depends on the true posterior or the exact forward model. As the main work is spent for reaching higher accuracy, when the surrogate model is already a reasonable approximation, we will optimistically use $\alpha=1$.

\subsubsection{Posterior-weighted \texorpdfstring{$L^2 $}{L2}  distance} \label{sec:L2}
In this section, we consider the posterior-weighted $L^2$ distance as approximation error metric
\[
    \mc E^{L^2} (\mc D ) = \norm{y-\bar y_\mc D}^2_{L^2 \left (\Omega; \pi(\cdot \mid y^m) \right )} 
\]
 which can also be seen as the expected squared 2-norm surrogate error. 

To obtain a computable error model, we work similarly as done with the Kullback-Leibler divergence: first the expected value with respect to the true posterior is replaced by an expected value with respect to the approximate posterior, then as the forward model is assumed to be a realization of a GP the integrated function are averaged over the Gaussian process. This is summarized in the following two corollaries.

\begin{corollary}
    Let $\psi$ be defined as in~\eqref{eq:psi}. Then,
    \[
        \mc E^{L^2} (\mc D )\leq \int_\Omega \pi(p \mid y^m, \mc D) \norm{y(p) - \bar y (p) }_2^2 \exp(\psi(y^m,p)) \, dp 
    \]
    holds.
\end{corollary}
\begin{proof}
    As stated in the proof of Theorem~\ref{thm:dKL-bound},~\eqref{eq:log-ratio} implies $\pi(p\mid y^m) \le \pi(p\mid y^m, \mc D)\exp(\psi(y^m,p))$ so the statement follows.
\end{proof}
\begin{corollary}
    let $y(p) \sim \mc N( \bar y(p),\Gamma(p))$. Then, the mean deviation is 
    \[
        \mathbb{E}\left[ \|\bar y(p) - y(p)\|_2^2\right] = \trace(\Gamma(p)),
    \]
    holds.
\end{corollary}
\begin{proof}
    The statement follows from Lemma~\ref{lem:surrogate-error-mean}, with $\Sigma = I_m$.
\end{proof}
By the above statements, for $\bar \psi (p)$ as in~\eqref{eq:dKL-model}, we obtain the error model
\begin{equation} \label{eq:L2-model}
    E^{L^2}(\mc D) = \int_\Omega \pi (p\mid y^m,\mc D) e^{L^2}_\mc D(p) \, dp, 
    \quad \text{with} \quad
        e^{L^2}_\mc D(p) =  \trace(\Gamma(p)) \exp(\bar\psi(p)),
\end{equation}
which is a GP-averaged estimate of the upper bound of $\mc E^{L^2} (\mc D) $ and can be numerically computed at run time.

\subsection{Work model} \label{sec:work-model}

Evaluating the forward model $y$ with a prescribed tolerance $\tau_i$ incurs a computational cost $W$ that depends on the problem, the simulation methodology, the evaluation position $p_i$, and the tolerance $\tau_i$. We assume that the cost can be written as a function of $\tau_i$ only, and in particular does not vary with the position $p_i$.

\begin{remark}
    While it is reasonable to assume that problem and simulation method, once chosen, affect all evaluations in the same way, the independence of work of the position is a simplification. In PDE problems, local solution structures like corner singularities or turbulence that can show up for one parameter value, but not a different one, often require different mesh widths or time step sizes, and thus can affect the computational effort for model evaluation significantly. However, this impact is hard to quantify a priori, though it could in principle be learned during the surrogate model building, and we restrict the attention to the simpler case.
\end{remark}

The actual computational work for an evaluation is not available prior to the evaluation, such that we need to rely on estimates. For finite element methods, a number of established asymptotic estimates for the computational work depending on the tolerance exist. Following~\cite{SemlerWeiser2023,WeiserGhosh2018}, we consider adaptive finite elements of degree $r$ in spatial dimension $l$ and an optimal solver such as multigrid, which yields the asymptotic work
\begin{equation} \label{eq:work-model}
    W(\tau_i) = \tau_i^{-l/r} .
\end{equation}
Being asymptotic for $\tau_i\to 0$, this estimate is quite rough for low accuracy evaluations, but are quite accurate for more expensive ones. 

The total work for a design $\mc D$ is just the sum of the efforts for all the evaluation points, i.e.
\[
    W(\mc D) = \sum_{(p, \tau) \in \mc D} W(\tau).
\]

Given a design $\mc D$, we are interested in refining such design by adding some new points and improving the accuracies of already included points. We say that $\mc D'$ refines $\mc D$, noted $\mc D' \le \mc D$, if $\mc D'$ contains all the points in $\mc D$ with better or equal accuracy, i.e. for any $(p, \tau) \in \mc D$ there exists $\tau' \leq \tau$ such that $(p, \tau') \in \mc D'$. Consequently, we assume that evaluations are stored and can be continued, which implies that, given an already evaluated $\mc D$, the computational cost of evaluating a refinement $\mc D'$ is $W(\mc D' \mid \mc D) = W(\mc D') - W(\mc D)$.

\begin{remark}
    Note that, in the case of FE evaluations of the model, in order to be able to update any of the points already in the design $\mc D$, both the final solution and the final grid need to have been saved after the previous evaluation. This is quite natural in the context we are facing: since the aim is including points where the density of the posterior is significant, we can expect the behavior of the underlying system to be similar in most of the points. This makes saving and re-using grids efficient and effective.
\end{remark}

As discussed in \cite[Sect. 3.3.2]{SemlerWeiser2023}, $\frac{l}{r}$ determines the domain of admissible evaluation precisions and impacts the existence of a unique minimizer in a problem analogous to~\eqref{eq:doe}, with a different target function.

\subsection{Design of computer experiments}\label{sec:doe}

With error and work models at hand, choosing the most efficient design reduces to solving the minimization problem~\eqref{eq:doe}. 
One difficulty in solving the problem is that the impact of design changes on the error estimate $E$, which is necessarily estimated based on the current design, will not be captured particularly well as long as the surrogate is not approximating the forward model sufficiently well. We thus resort to a greedy sequential heuristic for solving~\eqref{eq:doe} by partitioning the computational budget in $\Delta W_j$ for $j=1,\dots, J$ such that $\sum_{j=1}^J \Delta W_j = W$, and then spending it incrementally:
\begin{equation} \label{eq:incremental-doe}
    \min_{\mc D_{j} \le \mc D_{j-1}} E(\mc D_{j}) \quad \text{s.t.} 
    \quad W(\mc D_{j} \mid \mc D_{j-1}) \le \Delta W_j, \quad j=1,\dots,J.
\end{equation}
Here, $\mc D_{j-1}$ is the current design, $\mc D_{j}$ the one to be determined in the current iteration $j$, and $\mc D_0$ is any initial design.

\begin{remark}
    The sequential formulation of the training problem allows for interleaved sampling: instead of first training the GPE and then sampling the posterior, we can sample the posterior at each step of the training problem, accumulating samples in a sliding-window manner.
\end{remark}

Still, the sequential problems are highly nonlinear and non-convex mixed combinatorial-continuous optimization problem. Solving each of them accurately incurs likely more computational effort than can be saved by the fewer and less expensive forward model evaluations. Thus, we separate the selection of evaluation points $p_i$ from the selection of tolerances $\tau_i$.

\subsubsection{Evaluation points} \label{sec:pos-prob}
At each iteration $j$, we first select promising parameter positions to be added into the training set. Random sampling from the approximate posterior $\pi( \cdot \mid  y^m,\mc D_{j-1})$ is sufficient to guarantee convergence~\cite{HelinStuartTeckentrupZygalakis2023}. To obtain better candidates, we seek points where spending computational work promises to be particularly useful. As in~\cite{VillaniUngerWeiser2024,SemlerWeiser2024}, we estimate the impact of spending computational budget in one point by the sensitivity of the global error $E$ with respect to the spending of computational work in that position. 

For a parameter point $p \in \Omega$, let $\tau_p = \frac{\sum_{i=1}^{m} \sqrt{\Gamma_{i,i}(p)}}{m}$ be the mean of the predictive standard deviation in $p$. To compute the sensitivity of $E(\cdot)$ with respect to computational work in $p$, we assume as a baseline value the training set $\mc D_{\tau_p} = \mc D_{j-1} \cup \{ (p, \tau_p) \} $ resulting from the addition of $(p, \tau_p)$ as the next training point. 
Theoretically, it would be sufficient to consider 
\begin{equation} \label{eq:pos-prob-theo}
    \frac{\partial E(\mc{D}_{\tau_p})}{\partial {\tau_p}} \frac{\partial {\tau_p}}{\partial W}
\end{equation}
to obtain the desired sensitivity. Practically, this quantity is unavailable for computation due to two problems: 
\begin{itemize}
    \item computing the error model $E(\mc{D}_{\tau_p})$ involves integration with respect with the approximated posterior $\pi(p \mid y^m, \mc D_{\tau_p}) $ distribution: evaluating its derivative with respect to $\tau_p$ would be costly and demanding from a numerical point of view.
    \item as the predictive mean $\bar y$ depends on the data points $Y_\tau$, how adding a new model evaluation $y_{\tau_p}(p)$ affects $\bar y$ cannot be predicted before the actual model evaluation.
\end{itemize}  

To deal with the first problem, we substitute $\pi(p \mid y^m, \mc D_{\tau_p}) $ with the available posterior $\pi(p \mid y^m, \mc D_{j-1}) $: besides its practical necessity, as it allows us to interchange derivative and integral and reduces greatly the costs of integration through sampling techniques, the substitution can be justified by assuming a strong similarity between the approximated posteriors resulting from the addition of a single point. This results in the approximation \[
    \frac{\partial E(\mc{D}_{\tau_p})}{\partial {\tau_p}} \approx \int_\Omega \pi(p' \mid y^m, \mc D_{j-1}) \frac{\partial e_{\mc D_{\tau_p}}(p')}{\partial \tau_p}  \, dp' .
\]

To solve the second problem, we fix the predictive mean to the mean of $y_{\mc D_{j-1}}$ and consider only the impact of the change in variance, which is computable a priori. The resulting approximation is \[
\frac{\partial e_{\mc D_{\tau_p}}(p')}{\partial \tau_p} \approx \frac{\partial e_{\mc D_{\tau_p}}(p')}{\partial \Gamma(p')}\frac{\partial \Gamma(p')}{\partial \tau_p} .
\]

These approximations provide an estimate of~\eqref{eq:pos-prob-theo} which, as the derivatives can be computed analytically and are explicitly formulated in Appendix~\ref{app:derivatives}, can be evaluated at run time with little cost:
\begin{equation} \label{eq:acquisition}
    \mc A(p) = \int_\Omega \pi(p' \mid y^m, \mc D_{j-1}) \frac{\partial e_{\mc D_{\tau_p}}(p')}{\partial \Gamma(p')}\frac{\partial \Gamma(p')}{\partial \tau_p}   \ dp' \ \frac{\partial {\tau_p}}{\partial W}.
\end{equation}
We adopt $\mc A $ as an acquisition function, and select a number of its local maximizers as candidate points. 

\subsubsection{Evaluation tolerances} 

Having obtained a set of candidates $ \mc C = \{p_{s+1}, \dots,  p_{s+r}\}$, we need to select which of these to add to the experimental design as well as to optimize the tolerances of both the newly included points and the points already in the previous design $\mc D_{j-1} = \left( (p_i,\tau_i)\right)_{i=1,\dots,s}$.

For a candidate $p_i$ excluded from the training set, we write the corresponding tolerance to be infinite, $\tilde \tau_i = \infty$. This way, the new design $\mc D_j$ is determined by the assignment of a (possibly infinite) tolerance $\tilde \tau_i$ to each point in $\{ p_1, \dots, p_{s+r} \}$ and the choice of a new design is equivalent to the selection of one element of \begin{align*}
    \mc T_j = \Big \{ T = (\tilde \tau _1, \dots, \tilde \tau _{s+r}) \in ( \R ^+ \cup \{ +\infty \} ) ^{s+r} \mid \tilde \tau_i & \leq \tau_ i \text{ for } i \leq s \\
    & \text{and } \ \sum_{i=1}^{s+r} W(\tilde \tau_i ) \leq \Delta W_j  + W(\mc D) \Big \}.
\end{align*}

To be able to numerically solve Problem~\eqref{eq:incremental-doe} as a function of the tolerances, we face similar problems to the one we faced for the selection of candidates~\eqref{eq:pos-prob-theo}. 

Computing $E(\mc D_j)$ involves integration with respect with the approximated posterior $\pi(p \mid y^m, \mc D_{j}) $: doing so would make the solution of the minimization problem as a function of the accuracies practically challenging, as approximating the integral would require considerable computational effort for each evaluation of $E(\cdot)$. Again, we instead integrate with respect to $\pi(p \mid y^m, \mc D_{j-1})$.

Moreover, how the evaluation of $y$ within a certain tolerance affects $\bar y$ cannot be predicted without actual model evaluations. As a consequence, we again fix the predictive mean to the mean of $y_{\mc D_{j-1}}$.
By writing $e_{\mc D(T)}$ for the local error indicator as a function of the predictive variance $\Gamma(\cdot)$ corresponding to the tolerances $T \in T_j$ assigned to the candidates $\{ p_1, \dots, p_{s+r} \}$, the above approximations result in the following formulation:
\begin{equation} \label{eq:tolerance-prob}
     \min_{T \in \mc T_j} \int_\Omega \pi(p \mid y^m, \mc D_{j-1}) e_{\mc D (T)} (p) \, dp.
\end{equation}
As the derivatives of the target function are the ones involved in the acquisition function~\eqref{eq:acquisition} and described in Appendix~\ref{app:derivatives}, gradient information about the target is available and the problem can be treated numerically: the next section describes the computational aspects and provides a general outline of the training strategy.

\subsection{Posterior sampling with adaptively trained surrogate}\label{sec:algo}

The previous section introduced the acquisition function $\mc A$~\eqref{eq:acquisition} and the tolerances optimization problem~\eqref{eq:tolerance-prob} in order to render it possible to treat the sequential design of experiments problem~\eqref{eq:incremental-doe} which constitutes the core of the training strategy. 

At iteration $j$, the current posterior $\pi(p \mid y^m, \mc D_{j-1}) $ is sampled and a number of the oldest samples from preceding iterations is discarded, to obtain the chain of available samples $\mc S_j$. Conserving part of the samples is especially efficient after a few iterations, when the posterior does not change drastically, and in settings where the cost of sampling $\pi(p \mid y^m, \mc D_{j-1}) $ is not insignificant when compared with the cost of model evaluation.

As the selections of both the candidate points and the evaluation tolerances involve the estimation of a great number of integrals involving the currently available posterior approximation $\pi(p \mid y^m, \mc D_{j-1}) $, the adoption of MCMC integration~\eqref{eq:MC-integration} exploiting the samples $\mc S_j$ is greatly efficient. 

The proposed sampling strategy can be summarized by the following algorithm, and Figure~\ref{fig:2dsteps} provides a visualization of the strategy's iterations.
\phantomsection \makeatletter\def\@currentlabel{Algorithm 1}\makeatother
\label{algo:AGP} \textbf{Algorithm 1.}  Be given an initial design $\mc D_0$, empty sample chain $\mc S_0$ and a fractionating of the budget $\Delta W_1,\dots, \Delta W_J$. 

For $j =1, \dots, J$ do: 
\begin{enumerate}
    \item \textbf{Sample the posterior.} Decide a number $n_j$ of new samples to draw, a number $h_j$ of old samples to discard. Discard the $h_j$ oldest samples from $\mc S_{j-1}$, concatenate the remaining with $n_j$ samples from $\pi(p \mid y^m, \mc D_{j-1}) $: the resulting chain of samples is $\mc S_j$.
    \item \textbf{Select the candidates.} For any $p \in \Omega$, approximate $\mc A(p)$ by \begin{equation} \label{eq:discrete-pos-prob}
        \mc A(p) \approx \frac{1}{|\mc S_j|} \Bigg ( \sum_{p' \in \mc S_j} \frac{\partial e_{\mc D_{\tau_p}}(p')}{\partial \Gamma(p')}\frac{\partial \Gamma(p')}{\partial \tau_p} \Bigg )  \frac{\partial {\tau_p}}{\partial W}
    \end{equation} and select a number $s_j$ of local maximizers as candidates.
    \item \textbf{Select the tolerances.} Solve the discretization 
    \begin{equation}\label{eq:discrete-tol-prob}
        \min_{T \in \mc T_j} \frac{1}{|\mc S_j|} \sum_{p \in \mc S_j} e_{\mc D (T)} (p)
    \end{equation} of Problem~\eqref{eq:tolerance-prob} to determine the new evaluation tolerances $T = (\tau_1, \dots \tau_{s_j+r_j})$. If $\tau_i = \infty$ for some $i = s_j +1, \dots , s_j+r_j$, exclude the corresponding candidate $p_i$ from the design. This results in the new experimental design $\mc D_j  = \big ( (p_i, \tau_i))_{i = 1, \dots , s_{j+1}}$
    \item \textbf{Evaluate the model.} For every $(p, \tau )$ in $\mc D_j$ which was updated, i.e. the new evaluation points and the old points where the tolerance has changed, evaluate the model $y_\tau(p)$. Having obtained the new values $(y_1, \dots, y_{r_{j+1}})$, update the training set and re-train the surrogate to obtain $y_{\mc D_j}$.
\end{enumerate}
At last, the final set of samples $\mc S_{J+1}$ is obtained by drawing $n_{J+1}$ samples from $\pi(p \mid y^m, \mc D_J)$ and discarding the oldest $h_{J+1}$ samples, for some numbers  $n_{J+1},h_{J+1}$. \medskip

\begin{figure}[H]
\makebox[\textwidth][c]{\includegraphics[height = 110pt]{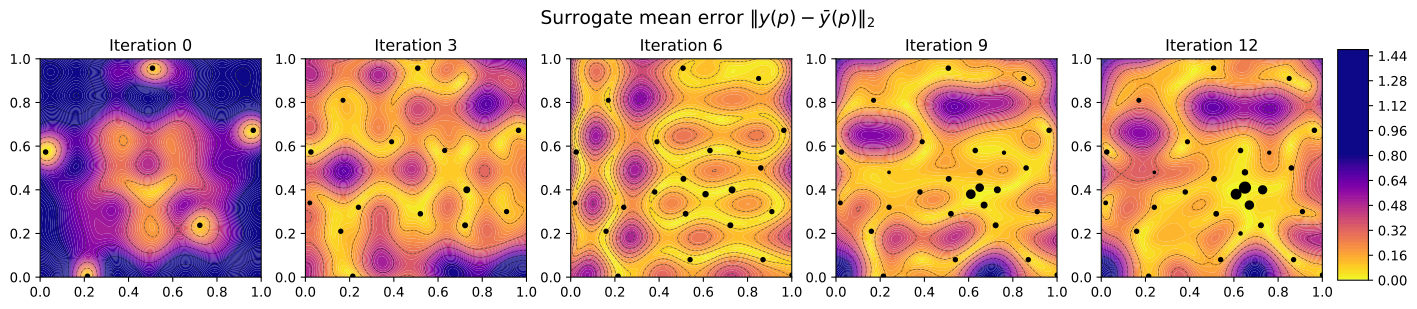}}
\vspace{-0.5cm}\caption{Surrogate mean error and training points (black) for different steps of a realization of~\ref{algo:AGP} under the experimental conditions described in Experiment~\ref{exp:2d}. The size of the each training point is proportional to the computational work spent in the point.}  \label{fig:2dsteps}
\end{figure}

Note that while local maximizers of~\eqref{eq:discrete-pos-prob} have to be found by a gradient-free optimization method, derivatives of the target in Problem~\eqref{eq:discrete-tol-prob} are available and gradient-based optimization methods can be utilized. In both cases, the problem is often non-convex and thus we resort to multi-start optimization.

\section{Numerical experiments} \label{sec:experiments}
To test \ref{algo:AGP}, a Python implementation has been tested on three different forward models, running on an Intel Core i7-9700T × 8 CPU. The first forward model is purely synthetic and does not have any underlying physical system, while the second and the third are based on the diffusion and the Poisson equation respectively.

Multiple runs are performed with different settings and choices, comparing the performances under different conditions:
\begin{itemize}
    \item \textbf{Error metric.} The error models introduced in Section~\ref{sec:error-model} are considered and compared by evaluating each under both error metrics.
    \item \textbf{Model cost.} The work model introduced in Section~\ref{sec:work-model} depends on the discretization cost, thus different discretization costs are considered.
    \item \textbf{Budget fractionating.} As the sequential formulation~\eqref{eq:incremental-doe} requires a division of the computational budget, we fraction the available budget a priori. We consider constant and geometric division of the budget, resulting in the \texttt{AGP-const} and \texttt{AGP-geom} strategies respectively. 
    \item \textbf{Adaptivity level.} The adaptive strategy presented is compared with a fixed-tolerance training strategy, \texttt{posAGP}, which chooses evaluation points adaptively, but does not optimize the tolerances. As a benchmark, a Latin Hypercube Sampling training point generation \texttt{LHSGP} is also considered.  
\end{itemize}
To assess the interaction between these aspects, all different combinations are tested. Multiple simulated measurement are generated resulting in different target posteriors, for each measurement the strategy is run under each experimental condition multiple times and the results are then averaged. Averaging over repeated runs is necessary as the outcomes depend on the evaluations of the discrete model which, as explained in the next Section, has random behavior in our simulation.

The posterior is sampled in a sliding-window manner, keeping part of the samples from the previous iteration. Both the number of new and of discarded samples grow quadratically with the number of iterations: the precise numbers for each forward model are described in the subsection dedicated.

\subsection{Implementation details}

In all the three cases, the forward model $y$ is analytically available and thus no discretization is necessary: the discretization error at tolerance $\tau$ is thus simulated by adding zero-mean Gaussian noise $\mc N (0, \tau I_m)$ to a forward model evaluation $y(p)$. The choice of analytical forward models has been taken as it allows to evaluate the actual error and compare convergence rates, reduces the time required by the simulation, and renders it possible to test different model costs without having to implement different numerical models. 

The Gaussian process surrogate model uses a GPyTorch~\cite{GPyTorchPaper} base model, and considers a separable Gaussian kernel~\eqref{eq:gauss-kernel} with diagonal $K_c$, modeling independent response components with individual scaling. Hyperparameter-tuning is performed using GPyTorch's Adam optimizer on the marginal likelihood, with Automatic Relevance Determination and lengthscale prior $Gamma(1,10)$ to optimize those hyperparameters determining the norm $\norm{\cdot}_k$. The model is set to work with double-precision floating-point numbers.

To sample the posterior, we consider the emcee~\cite{emceePaper} implementation of Ensemble Sampling, a MCMC technique introduced by~\cite{GoodmanWeare} which guarantees invariancy with respect to affine transformation of the parameter space by utilizing multiple correlated chains. The effectiveness of sampling crucially depends on the proposal function, which in the case of Ensemble samplers are also known as moves. As the posterior distribution could potentially exhibit multimodal behavior, we adopt the Differential-Independence Mixture Ensemble move presented by~\cite{Boehl}, which adapts to Ensemble sampling the Differential Evolution MCMC framework introduced by~\cite{TerBraak} and behaves efficiently on multimodal distributions. 

The two discrete optimization problems to determine evaluation accuracies and tolerances are solved through the SciPy optimizer \texttt{scipy.optimize.minimize}. Local maxima of~\eqref{eq:discrete-pos-prob} are found using the L-BFGS-B algorithm~\cite{ZhuBirdNocedal}. Problem~\eqref{eq:discrete-tol-prob} is solved by applying a coordinate transformation in order to linearize the work constraint and then by utilizing the Sequential Least Squares Quadratic Programming algorithm. In both cases, we adopt a multi-start approach by selecting a number of initial points through a low-discrepancy sequence. 

\subsection{2D parameter space - synthetic example}\label{exp:2d}

The first example considers a two dimensional parameter space $\Omega = [-\frac{1}{2},\frac{1}{2}]^2 $ with a flat prior and a three dimensional measurement space, $\R^3$. The forward map is purely synthetic 
\[
y(x,y) = L(x,y) + \varphi(x,y)\begin{bmatrix} 1,1,1 \end{bmatrix}^T, \]
and given by the sum of a linear term $L$ and a scalar oscillatory term $\varphi$
\[
    L(x,y) = \Big [\sin{(k)} + \cos{(k)},  \ \sin{(k)} - \cos{(k)}  \Big ]_{k=0,10,20}\begin{bmatrix} x, y \end{bmatrix}^T, \]
    \[
    \varphi(x,y) = \frac{1}{10}\big( \sin(20x-2) +\sin(20y-2)  \big).
\]
\begin{figure}[H] 
\begin{centering}
\includegraphics[height = 110pt]{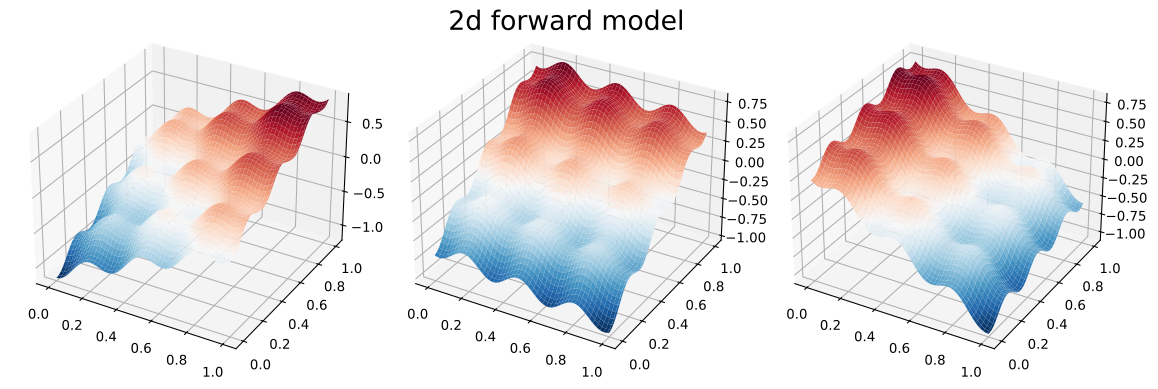}
\caption{Ground truth forward model for the 2-d parameter space.} \label{fig:2Dtruth}
\end{centering}
\end{figure}
We consider 5 different simulated values for $y^m$ by drawing on the parameter space via Latin Hypercube Sampling, evaluating $y$ and adding normal noise $\mc N (0, \Sigma^2)$ for $\Sigma= 0.02\cdot I_3$. 

For the initial design $\mc D_0$, five initial evaluation points are selected by Latin Hypercube Sampling with default tolerance $\tau = 0.05$. Every iteration we seek up to 3 local maximizers of the acquisition function as candidates and we perform $J=13$ iterations. The work model~\eqref{eq:work-model} is set to estimate the model cost for $\frac{l}{r} =1, 1.5, 2, 3$.

The assigned budget is $13\cdot 3 \cdot \tau^ {-\frac{l}{r}}$, amounting to 13 iterations with 3 points each with tolerance $\tau$ for the non-adaptive benchmark strategies \texttt{posAGP} and \texttt{LHSGP}, where no tolerance optimization is considered. For the adaptive training strategy, a constant budget division of $3 \cdot \tau^ {-\frac{l}{r}}$ per iteration is adopted for \texttt{AGP-const} and a geometric budget division of $\alpha^{j-1} \cdot \tau^ {-\frac{l}{r}}$ at iteration j, with $\alpha = 1.173$, is adopted for \texttt{AGP-geom}.

The chain of samples is built by sampling 1600 samples at the first iteration up to 16000 in the final one, while 1600 samples are discarded in the second iteration up to 8000 in the last. Precisely, at iteration $j$ the number $n_j$ of samples to draw and the number $h_j$ of old samples to discard are given by \[
n_j = 1600 + 14400 \left \lfloor\frac{j-1}{13} \right \rfloor^2 , \quad 
h_j = 1600 + 6400 \left \lfloor\frac{j-1}{13} \right \rfloor^2,
\]
where $\lfloor \cdot \rfloor$ denotes the floor of a real number. Figure~\ref{fig:2Dsamples} illustrates the gradual accumulation of samples.
\begin{figure}[H] 
\begin{centering}
\includegraphics[height = 80pt]{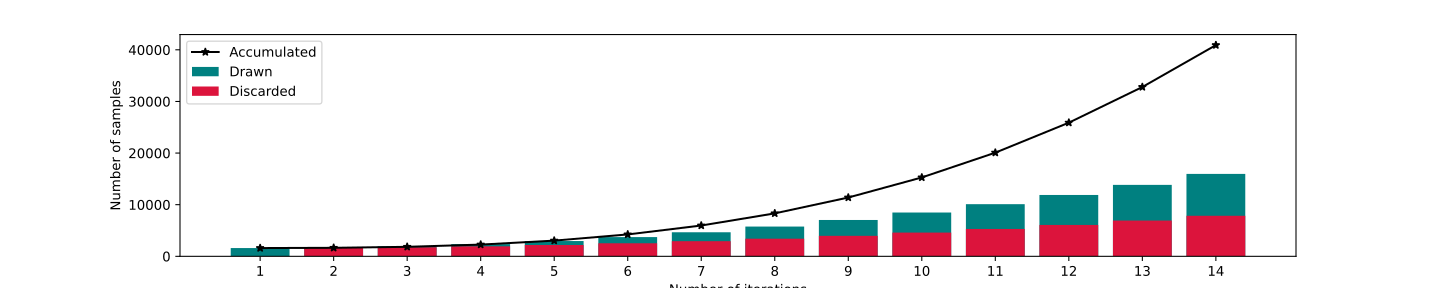}
\caption{Accumulated samples (black line), drawn samples (green bar), discarded samples (red bar) for the 2-d example.} \label{fig:2Dsamples}
\end{centering}
\end{figure}
Figure~\ref{fig:2dsteps} in Section~\ref{sec:algo} exemplifies a realization of \texttt{AGP-const} for work model with $\frac{l}{r} = 1$ and $L^2$ target metric. After an exploration of the parameter space in the first iterations, the selection of training points focuses on areas where the posterior places significant mass: fewer new points are added and the accuracy of relevant points is improved, resulting in a notable reduction of the surrogate error in the posterior region.

Figure~\ref{fig:2d-convergence} depicts the convergence rates for every experimental set-up of the 2-d. As repeated runs are performed to average on the simulated discretization error, the metric value for every run(thinner line), the average performance for each measurement set (intermediate thickness) and overall mean (thicker line) are plotted.
The results show a remarkable improvement of the convergence metric for the adaptive strategies compared to space-filling designs, while tolerance optimization is particularly effective for the Kullback-Leibler divergence metric and when the cost of a more accurate model evaluation is not significantly higher. 

Figure~\ref{fig:2dtols} shows the distribution of evaluation tolerances over points for all the runs of \texttt{AGP-geom} with target KL divergence: for an higher FE simulation cost most points have the default tolerance $\tau$, while if the relative cost of improving the evaluation's accuracy is low the tolerances have a wider distribution. This reflects the reduction of the effectiveness of tolerance optimization for relatively higher cost of accurate simulations.

\begin{figure}[H]
\begin{centering}
\includegraphics[height = 90pt]{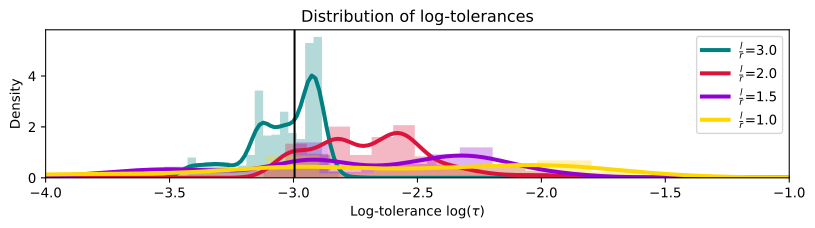}
\caption{Distribution of the logarithm of the evaluation tolerances for \texttt{AGP-geom} with target KL divergence and different simulation costs. The logarithm of the default tolerance is represented by the black vertical line. On average 34 points are included in the training set, resulting in around 850 different evaluation accuracies for each of the costs. The density function is reconstructed by \texttt{scipy.stats.gaussian\_kde}. 
}  \label{fig:2dtols}
\end{centering}
\end{figure}
\begin{figure}[H]
   \hspace{-1cm}
    \begin{minipage}[b]{0.1\textwidth}
    \centering
    $ \frac{l}{r}= 1 $ \\
    \vspace{2.4cm}
    $ \frac{l}{r}= 1.5 $ \\
    \vspace{2.4cm}
    $ \frac{l}{r}= 2 $ \\
    \vspace{2.4cm}
    $ \frac{l}{r}= 3 $ \\
    \vspace{1.8cm}
  \end{minipage}
  \begin{minipage}[b]{0.45\textwidth}
    \includegraphics[width=1.1\textwidth]{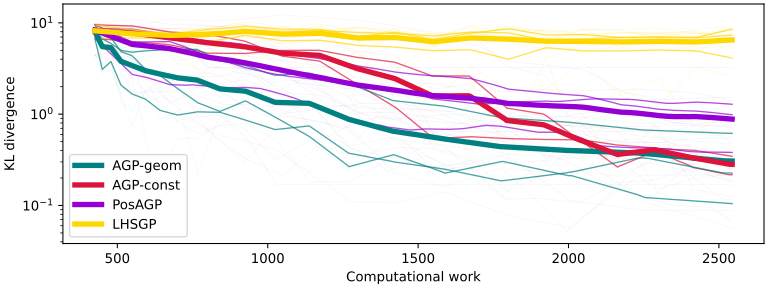}
    \includegraphics[width=1.1\textwidth]{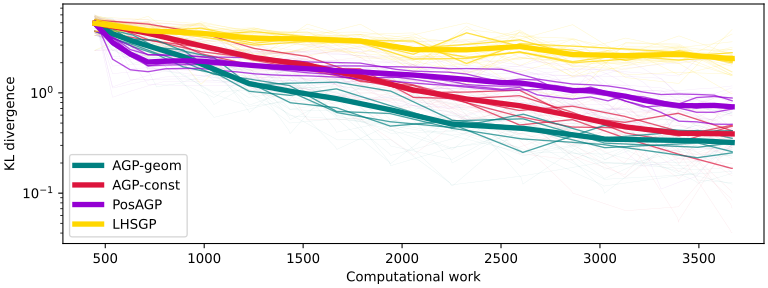}
    \includegraphics[width=1.1\textwidth]{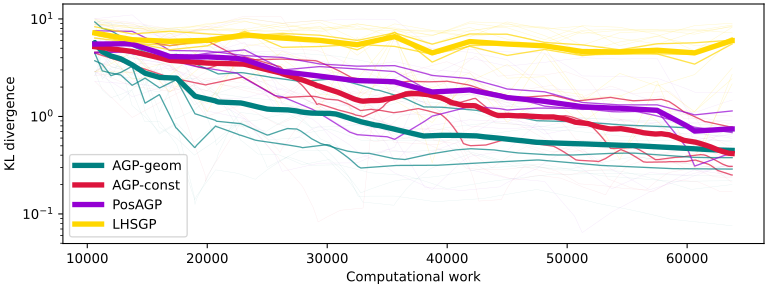}
    \includegraphics[width=1.1\textwidth]{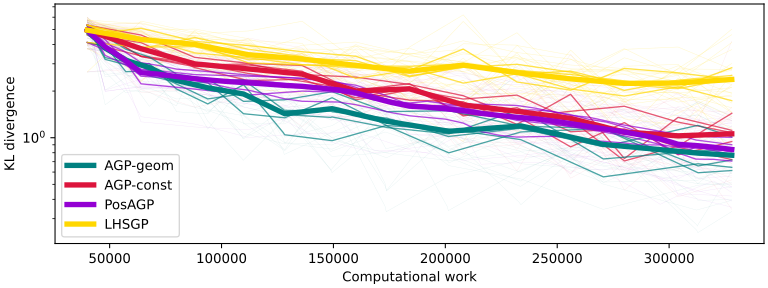}
    \centering
    Kullback-Leibler divergence.
  \end{minipage}
  \hfill
  \begin{minipage}[b]{0.45\textwidth}
    \includegraphics[width=1.1\textwidth]{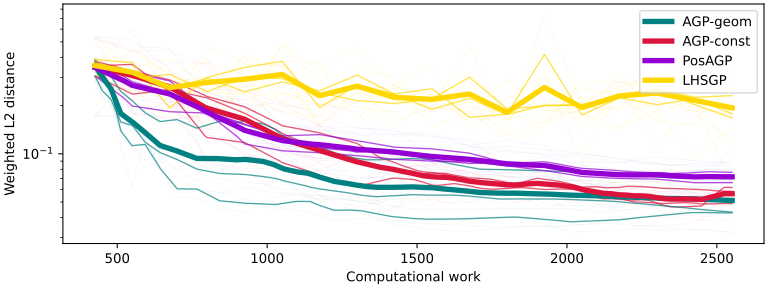}
    \includegraphics[width=1.1\textwidth]{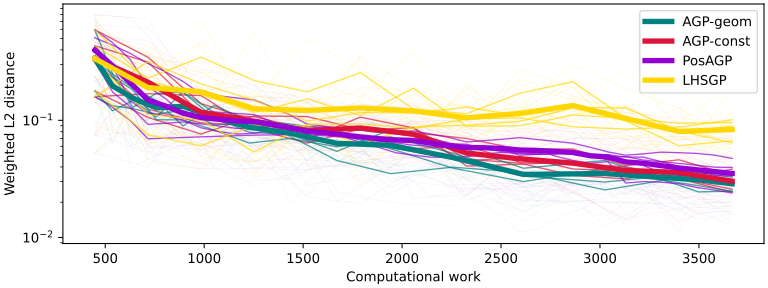}
    \includegraphics[width=1.1\textwidth]{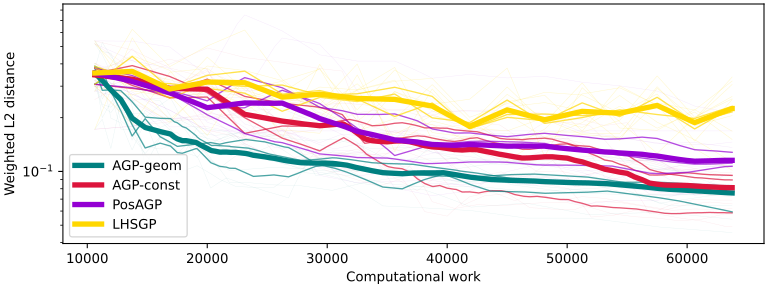}
    \includegraphics[width=1.1\textwidth]{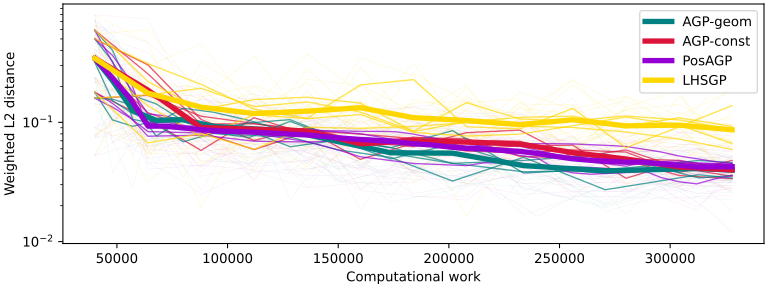}
    \centering
    Weighted $L^2$ distance
  \end{minipage}

  \caption{Convergence rates for the two different error models and four different FE simulation costs  on a 2-d parameter space. The thinner lines represent individual runs, the intermediate-thickness lines the average over 5 runs with the same measurements and posterior and the thicker lines the average on the whole 25 runs.}\label{fig:2d-convergence}
\end{figure}

\subsection{3D parameter space - diffusion equation}\label{sec:3d-exp}

The second numerical experiment involves the diffusion equation on $\R^3$
\[
u_t = \Delta u
\]
with initial conditions given by Dirac's delta centered in some $x_0$
\[
u (0,x; x_0) = \delta_{x_0}(x).
\]
We consider the fundamental solution of the diffusion equation as a solution of the above problem
\[
u(t,x; x_0) = (4 \pi t)^{-\frac{3}{2}} \exp \left( - \frac{\norm{x-x_0}_2^2}{4t}\right).
\]
We treat $x_0$ as an unknown parameter and want to identify it by measuring $u(t,x;x_0)$ for certain $t,x$. As a parameter space, $\Omega = [-1,1]^3$ is considered and a flat prior is assumed.

We consider three measuring times \[\mc Z =  \{ 0.5, 0.7, 1\}\] and six sensors positions given by the intersection of the cartesian axis with the unit sphere, \[\mc X=\{(\pm 1,0,0),(0,\pm 1,0),(0,0,\pm 1)\} \] and obtain an 18-dimensional measurement space $\R^{18}$. The measurement vector $y^m$ is then generated by adding centered Gaussian noise
\[
y_m \sim \mc N \left( \left( u(t^m,x^m;x_0)  \right)_{(t^m,x^m)\in \mc Z \times \mc X}, \Sigma^2 \right)
\]
with $\Sigma = 0.01 I_{18}$. 3 different simulated values of $y^m$ are generated and each of those is run under different experimental conditions 5 times. 

For the initial design $\mc D_0$, nine initial evaluation points are selected by Latin Hypercube Sampling with default tolerance $\tau = 0.02$. Every iteration we seek up to 4 local maximizers of the acquisition function as candidates and we perform $J =15$ iterations. For this experiment, we compare the work models with $\frac{l}{r} =1, 2$.

The assigned budget is $15\cdot 4 \cdot \tau^ {-\frac{l}{r}}$, amounting to 15 iterations with 4 points each with tolerance $\tau$ for \texttt{posAGP} and \texttt{LHSGP}. A constant budget division of $4 \cdot \tau^ {-\frac{l}{r}}$ per iteration and a geometric budget division of $\alpha^{j-1} \cdot \tau^ {-\frac{l}{r}}$ at iteration j, for $\alpha = 1.178$, are adopted for \texttt{AGP-const} and \texttt{AGP-geom} respectively.

The chain of samples is built by sampling 2400 samples at the first iteration up to 24000 after the final one, while 2400 samples are discarded in the second iteration up to 12000 after the last. At iteration $j$ the number $n_j$ of samples to draw and the number $h_j$ of old samples to discard are given by \[
n_j = 2400 + 21600 \left \lfloor\frac{j-1}{15} \right \rfloor^2 , \quad 
h_j = 2400 + 9600 \left \lfloor\frac{j-1}{15} \right \rfloor^2.
\]

\begin{table}
    \centering
    \begin{tabular}{cccc}
         &\texttt{AGP-geom}  & \texttt{AGP-const} & No tolerance optimization \\
         KL divergence & 43 & 39 & 69 \\
         $L^2$ distance & 29 & 28  & 69 \\
    \end{tabular}
    \caption{Average number of points in the training set of the 3-d experiment for different strategies and metrics, 25 runs each and $\frac{l}{r} = 1$.}
    \label{tab:3dpts}
\end{table}
In Table~\ref{tab:3dpts} the average number of training points for different experimental configurations is given. Strategies which optimize the evaluation tolerances include significantly less points than the non-adaptive counterparts, reducing the computational costs of evaluations of the surrogate model. The effect is particularly remarkable under the $L^2$ error model, where the number of points is more than halved.

Figure~\ref{fig:3d-convergence} depicts the convergence rates for every experimental set-up of the 3-d experiment. As for the 2-d case, single realizations (thinner line), average performance for each target posterior (intermediate thickness) and overall mean (thicker line) are plotted. The results are consistent with the ones obtained in the 2-d example, but in this experiment the $L^2$ metric is more affected accuracy optimization, possibly due to the more complicated forward model. Lower evaluation costs are again more sensible to accuracy optimization.

Figure~\ref{fig:3Dposterior} compares the posteriors resulting from a realization of the different training strategies. The adaptive algorithms have as a target metric the Kullback-Leibler divergence, while the work model is set to $\frac{l}{r} = 1$. All the GP-based posterior capture some aspects of the ground truth, but the fully adaptive \texttt{AGP} posteriors offer a better approximation for the same computational work.

\begin{figure}[H]
    \hspace{-1cm}
     \begin{minipage}[b]{0.1\textwidth}
     \centering
     $ \frac{l}{r}= 1 $ \\
     \vspace{2.4cm}
     $ \frac{l}{r}= 2 $ \\
     \vspace{1.8cm}
   \end{minipage}
   \begin{minipage}[b]{0.45\textwidth}
    \includegraphics[width=1.1\textwidth]{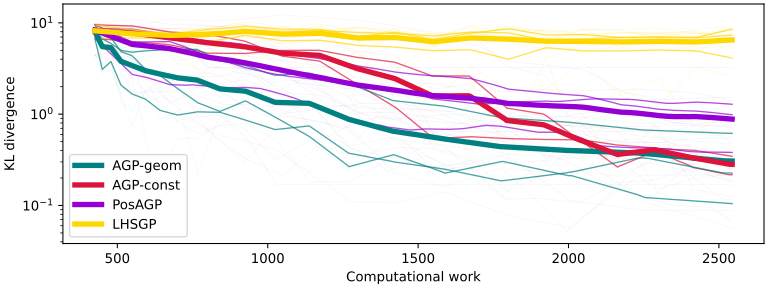}
    \includegraphics[width=1.1\textwidth]{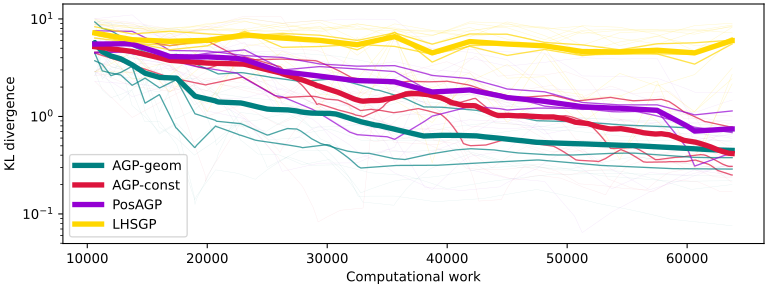}
    \centering
    Kullback-Leibler divergence.
  \end{minipage}
   \hfill
   \begin{minipage}[b]{0.45\textwidth}
    \includegraphics[width=1.1\textwidth]{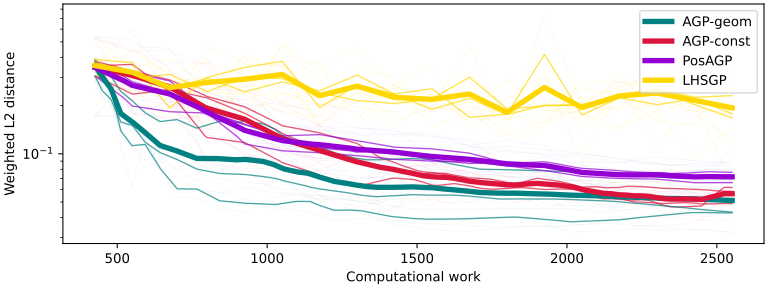}
    \includegraphics[width=1.1\textwidth]{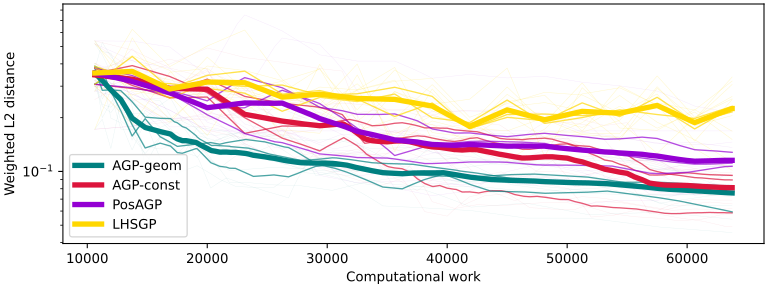}
     \centering
     Weighted $L^2$ distance
   \end{minipage}
 
   \caption{Convergence rates for the two different error models and two different FE simulation costs on a 3-d parameter space.}\label{fig:3d-convergence}
 \end{figure}

\begin{figure}[H] 
\begin{centering}
\includegraphics[height = 320pt]{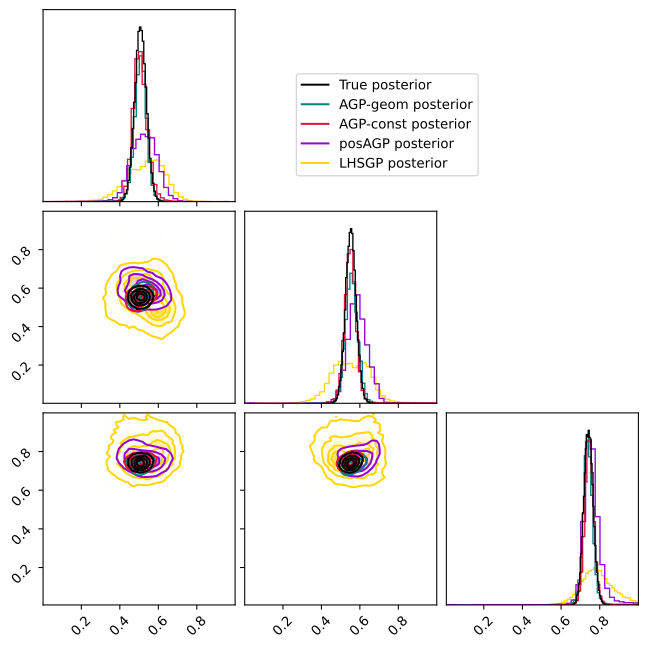}
\caption{Comparison of the samples from the different posteriors in one run of the 3-d experiment. Level lines correspond to $20\%,45\%,70\%, 90\%$ of the total number of samples. Plot generated using \texttt{corner.py}\cite{corner}.} \label{fig:3Dposterior}
\end{centering}
\end{figure}

\subsection{4D parameter space - Poisson equation}

The third numerical experiment involves the distributional Poisson equation on $\R^2$, with two Dirac's delta sources of opposite sign
\[
 \Delta u = \delta_{x_1}(x) - \delta_{x_2}(x)
\]
decaying at infinity
\[
\lim_{x \to \infty } u (x; x_1, x_2)= 0.
\]
As a solution, we consider the sum of Green's functions, omitting the $\frac{1}{2 \pi}$ constant:
\[
u(x; x_1, x_2) = - \log(\norm{x-x_1}_2) + \log(\norm{x-x_2}_2).
\]

We aim at identifying $x_1$ and $x_2$ by measuring $u(x;x_1, x_2)$ in 12 equally spaced positions on the unit sphere, 
\[\mc X = \left\{ \left(\cos\left(\frac{2 \pi i}{12}\right), \sin\left(\frac{2 \pi i}{12}\right) \right)    \right\}_{i = 0, \dots , 11 }. \]
Figure~\ref{fig:4Dmodel} provides a plot of $u(x;x_1, x_2)$ for $x_1 = [0.4, -0.67], x_2 = [-0.25, 0.67]$ and of the sensors in $\mc X$.

This results in a 12-dimensional parameter space $\R^{12}$ and in the measurement vector $y^m$ being generated by adding centered Gaussian noise
\[
y_m \sim \mc N \left( \left( u(x^m;x_1,x_2)  \right)_{x^m \in \mc X}, \Sigma^2 \right)
\]
with $\Sigma = 0.05 I_{12}$. 3 different simulated values of $y^m$ are generated and each of those is run under different experimental conditions 4 times. As a parameter space, $\Omega = [-1,1]^4$ with a flat prior is considered.

\begin{wrapfigure}{r}{0.5\textwidth}
\begin{centering} 
\includegraphics[width=0.48\textwidth]{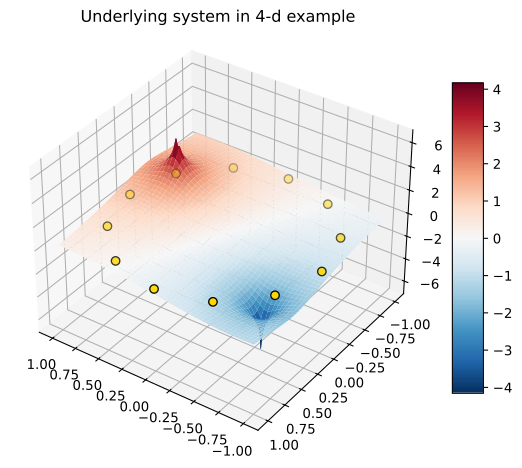} 
\caption{Considered solution and sensors.} \label{fig:4Dmodel}
\end{centering}
\end{wrapfigure}

For the initial design $\mc D_0$, 17 initial evaluation points are selected by Latin Hypercube Sampling with default tolerance $\tau = 0.04$. Every iteration we seek up to 5 local maximizers of the acquisition function as candidates and we perform $J= 20$ iterations. For this experiment, we compare the work models with $\frac{l}{r} =1, 2$.

The assigned budget is $20\cdot 5 \cdot \tau^ {-\frac{l}{r}}$, amounting to 20 iterations with 5 points each with tolerance $\tau$ for \texttt{posAGP} and \texttt{LHSGP}. A constant budget division of $5 \cdot \tau^ {-\frac{l}{r}}$ per iteration and a geometric budget division of $\alpha^{j-1} \cdot \tau^ {-\frac{l}{r}}$ at iteration j, for $\alpha = 1.148$, are adopted for \texttt{AGP-const} and \texttt{AGP-geom} respectively.

The chain of samples is built by sampling 3200 samples at the first iteration up to 32000 after the final one, while 3200 samples are discarded in the second iteration up to 16000 after the last. As the cost of sampling starts to be significant given the considerable number of runs we need to perform, samples are drawn and discarded every two iterations only. At iteration $j$ the number $n_j$ of samples to draw and the number $h_j$ of old samples to discard are given by \[
n_j = 3200 + 28800 \left \lfloor\frac{j-1}{20} \right \rfloor^2 , \quad 
h_j = 3200 + 11800 \left \lfloor\frac{j-1}{20} \right \rfloor^2
\]
if $j \equiv_2 1$ and $0$ else.
Figure~\ref{fig:4Dsamples} illustrates the gradual accumulation of samples.
\begin{figure}[H] 
\begin{centering} 
\includegraphics[height = 80pt]{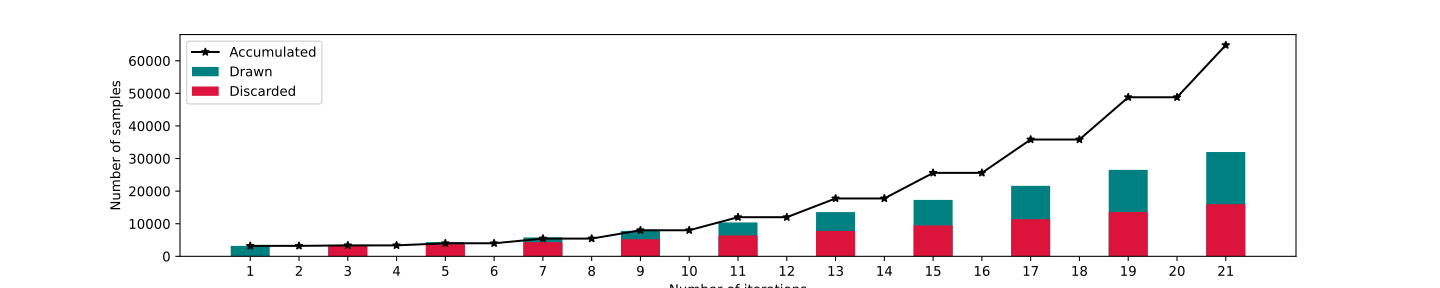} 
\caption{Accumulated samples (black line), drawn samples (green bar), discarded samples (red bar) for the 4-d example.} \label{fig:4Dsamples}
\end{centering}
\end{figure}

In Table~\ref{tab:4dpts} the average number of training points for different experimental configurations is given. The results are consistent with Experiment~\ref{sec:3d-exp}, with again a considerable reduction of training points from tolerance optimization. In this case, the gap between the two metrics is tighter.

Figure~\ref{fig:4d-convergence} depicts the convergence rates for every experimental set-up of the 4-d experiment. Again, single realizations, average performance for each target posterior and overall mean are plotted. The results are consistent with the ones obtained in the other experiments. In this case, tolerance optimization seems to be effective even for higher costs models. Moreover, \texttt{AGP-geom} obtains quickly a good accuracy level which is, on average, matched by \texttt{AGP-const} only in the final iterations. A possible explanation for both these behaviors lies in the higher dimension of the problem, as more points, more iterations, and a greater budget are considered.

\begin{table}
    \centering
    \begin{tabular}{cccc}
         &\texttt{AGP-geom}  & \texttt{AGP-const} & No tolerance optimization \\
         KL divergence & 84 & 81 & 117 \\
         $L^2$ distance & 74 & 78  & 117 \\
    \end{tabular}
    \caption{Average number of points in the training set of the 4-d experiment for different strategies and metrics, 25 runs each and $\frac{l}{r} = 1$.}
    \label{tab:4dpts}
\end{table}

\begin{figure}[H]
    \hspace{-1cm}
     \begin{minipage}[b]{0.1\textwidth}
     \centering
     $ \frac{l}{r}= 1 $ \\
     \vspace{2.4cm}
     $ \frac{l}{r}= 2 $ \\
     \vspace{1.8cm}
   \end{minipage}
   \begin{minipage}[b]{0.45\textwidth}
    \includegraphics[width=1.1\textwidth]{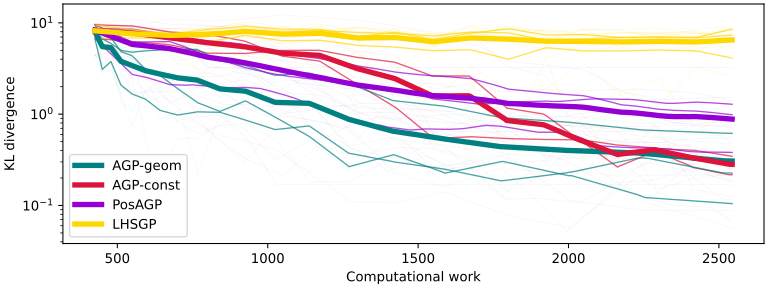}
    \includegraphics[width=1.1\textwidth]{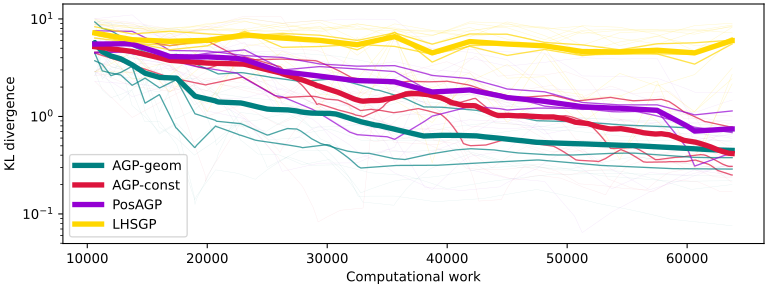}
    \centering
    Kullback-Leibler divergence.
  \end{minipage}
   \hfill
   \begin{minipage}[b]{0.45\textwidth}
    \includegraphics[width=1.1\textwidth]{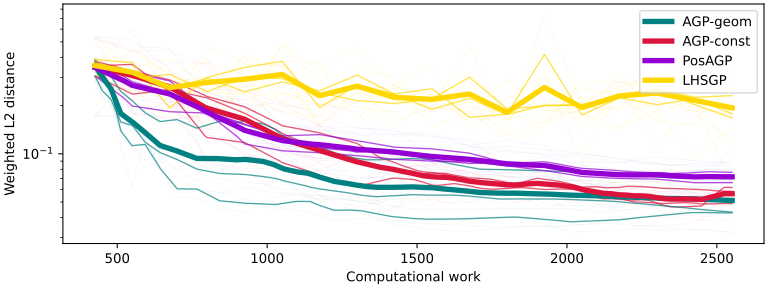}
    \includegraphics[width=1.1\textwidth]{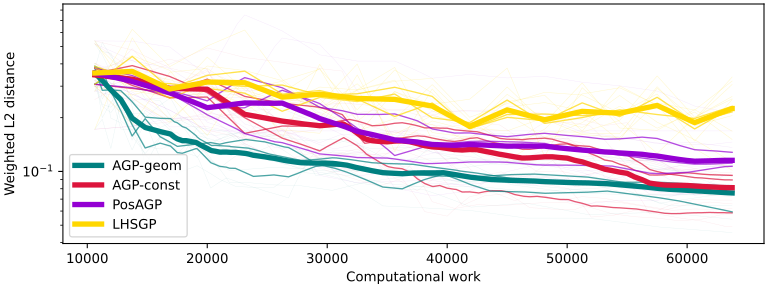}
     \centering
     Weighted $L^2$ distance
   \end{minipage}
 
   \caption{Convergence rates for the two different error models and two different FE simulation costs on a 4-d parameter space.}\label{fig:4d-convergence}
 \end{figure}

\section{Conclusions}
By interleaving the selection of training points with the accumulation of samples of the posterior, the proposed methodology provides a computationally efficient solution strategy for inverse problems. As demonstrated by the provided experiments, the optimization of the evaluation tolerances in addition to adaptive training point selection can significantly improve the performance of surrogate-based sampling whenever the cost of a more accurate simulation scales accordingly. Moreover, re-optimizing the accuracy of the training data further reduces computational costs by including fewer points in the training set. 

\paragraph{Acknowledgements.} This work has been funded by the German Research Foundation (DFG) under grant 501811638 within the project C07 of the priority program SPP 2388 "Hundert plus – Verlängerung der Lebensdauer komplexer Baustrukturen durch intelligente Digitalisierung".

\bibliographystyle{plain}
\bibliography{main}

\appendix
\section{Target function derivatives} \label{app:derivatives}
This appendix contains the analytical expression of the derivatives contained in the acquisition function~\eqref{eq:acquisition} and used for the gradients of the target function in Problem~\eqref{eq:tolerance-prob}.

\subsection{Derivative of the error model with respect to the variance prediction.}

Both the error models~\eqref{eq:dKL-model},\eqref{eq:L2-model} involve the function $\bar \psi$~\eqref{eq:acquisition} defined at the end of Section~\ref{sec:dKL}. Its derivative with respect to one variance component $\Gamma_{i,j}(p)$ is \[
\frac{\partial \bar \psi(p)}{\partial \Gamma_{i,j}(p)} = \left(  1 + \frac{\|y^m-\bar y(p)\|_{\Sigma^{-2}}}{ 2\sqrt{\trace\left(\Sigma^{-2}\Gamma(p) \right)} } \right) \left(\Sigma^{-2}\right)_{i,j}.
\] 
Thus, for the Kullback-Leibler error model the Jacobian is given by \[
    \frac{\partial e^{KL}_{\mc D_{\tau_p}}(p)}{\partial \Gamma(p)} = \exp(\bar\psi(p)) \left(  1 + \frac{\|y^m-\bar y(p)\|_{\Sigma^{-2}}}{ 2\sqrt{\trace\left(\Sigma^{-2}\Gamma(p) \right)} } \right) (1 + \bar\psi(p) ) \Sigma^{-2}
\]
and for the $L^2$ metric error model by 
\[
    \frac{\partial e^{L^2}_{\mc D_{\tau_p}}(p)}{\partial \Gamma(p)} = \exp(\bar\psi(p)) I_m + \trace(\Gamma(p)) \exp(\bar\psi(p)) \left(  1 + \frac{\|y^m-\bar y(p)\|_{\Sigma^{-2}}}{ 2\sqrt{\trace\left(\Sigma^{-2}\Gamma(p) \right)} } \right) \Sigma^{-2}.
\]
\subsection{Derivative of the variance prediction with respect to the evaluation tolerance.}

We consider a design $\left( (p_i,\tau_i)\right)_{i=1,\dots,s}$ and a GP kernel $k$. Using the explicit formula for the predictive variance given in~\cite[equation 2.24]{RasmussenWilliams2006}, the derivative of $\Gamma(p)$ with respect to $\tau_i$ is given by \[
    \frac{\partial \Gamma(p)}{\partial \tau_i} = 2\tau_i K_* (K + \mathfrak{T} ^2  )^{-1} (I_m \otimes \text{diag}(e_i))(K + \mathfrak{T} ^2  )^{-1} K_*^T,
\] 
where $\underline{e_1}, \dots,\underline{e_s}$ is the canonical basis of $\R^s$, and the kernel matrices are\[
K = \begin{bmatrix} k(p_1,p_1) & \dots & k(p_1,p_s) \\
    \vdots               & \ddots & \vdots \\
    k(p_s,p_1) & \dots & k(p_s,p_s)
\end{bmatrix} \ \text{ and } \
K_* =  \begin{bmatrix}
    k(p_1,p) \\
    \vdots \\
    k(p_s,p)
\end{bmatrix}.
\]

\subsection{Derivative of the evaluation tolerance with respect to the computational work.}
The derivative is obtained by considering the work model~\eqref{eq:work-model} and inverting it:
\[
\frac{\partial \tau}{\partial W} = -\frac{r}{l} W^{-\frac{r}{l}-1}.
\]

\end{document}